\documentclass{article}
\usepackage{amsmath}
\usepackage{amsfonts}
\usepackage{amsthm}
\usepackage{amssymb}
\usepackage[french]{babel}
\usepackage[latin1]{inputenc}
\usepackage[T1]{fontenc}

\newtheorem{thm}{Th\'eor\`eme}
\newtheorem{defi}{D\'efinition}
\newtheorem{cor}{Corollaire}

\newtheorem{lem}{Lemme}
\newtheorem{prop}{Proposition}
\newtheorem{rem}{Remarque}

\DeclareMathOperator{\Hom}{Hom}

\input xy
\xyoption{all}

\author{Niels Borne \and Michel Emsalem}

\title{Note sur la d\'etermination alg\'ebrique du groupe fondamental pro-r\'esoluble
d'une courbe affine}

\begin{document}
\maketitle
\section{Introduction}

On consid\`ere une courbe projective connexe et lisse $\bar X$ de genre $g$ sur un corps alg\'ebriquement clos $k$ de caract\'eristique 
$p$ (\'eventuellement nulle) et $X= \bar X \setminus \{ a_1, \dots , a_r\} $ la courbe affine d\'efinie sur $k$  obtenue en \^otant 
$r$ points \`a $\bar X$ (avec $r\geq 1$). La connaissance du groupe fondamental alg\'ebrique $\pi _1 ( X, \bar x)$ 
repose sur des arguments transcendants, en particulier sur le ``Th\'eor\`eme d'existence de Riemann`` (voir \cite{SGA1} XII Th\'eor\`eme 5.1). Le but de cette note est de prouver le th\'eor\`eme de structure pour le plus grand quotient pro-r\'esoluble et premier \`a $p$ du groupe fondamental, en n'utilisant que des outils alg\'ebriques. On reprend \`a cette fin les id\'ees essentielles et les arguments de \cite{JPS}.

Pour un groupe fini $G$, on note $n_G$ le nombre minimal de
g\'en\'erateurs de $G$, et ${\cal P}_G$ la propri\'et\'e suivante :

$$n_G\leq 2g+r-1 \iff \exists \pi_1(X,\overline{x})\twoheadrightarrow G$$

On sait que ${\cal P}_G$ est vraie pour tout groupe $G$ d'ordre premier \`a la caract\'eristique $p$, mais la
preuve est ``transcendante''.

On prouvera avec des outils alg\'ebriques l'\'enonc\'e suivant.

\begin{prop}
\label{state}  
Soit la suite exacte de groupes :
$$1 \rightarrow A \rightarrow G \rightarrow H\rightarrow 1 \; .$$
Si $A$ est \emph{r\'esoluble}, si $G$ est d'ordre premier \`a $p$, et si ${\cal P}_H$ est vraie, alors ${\cal P}_G$ est vraie. 
\end{prop}

En particulier, on a donc une preuve alg\'ebrique du fait que ${\cal P}_G$
est vraie pour $G$ r\'esoluble d'ordre premier \`a la caract\'eristique $p$.

On en d\'eduit une preuve alg\'ebrique du th\'eor\`eme de structure du quotient $p'$-r\'esoluble du groupe fondamental alg\'ebrique d'une courbe alg\'ebrique affine. Pour un groupe (profini) $G$, notons  $G^{res}$ la limite projective de ses
quotients r\'esolubles finis d'ordre premier \`a $p$. Pour tout entier $N \geq 1$, $F_N$ d\'esignera 
un groupe profini libre \`a $N$ g\'en\'erateurs.

\begin{thm} \label{structure}

Soit $X= \bar X \setminus \{ a_1, \dots , a_r\} $ une courbe affine d\'efinie sur un corps alg\'ebriquement clos $k$ de caract\'eristique 
$p$ (\'eventuellement nulle), 
o\`u $\bar X$ d\'esigne une courbe projective lisse et connexe de genre $g$ et $r $ un entier sup\'erieur ou \'egal \`a $1$.  Alors :

$$ \pi_1^{res}(X,\overline{x})\simeq F_{2g+r-1} ^{res} \; .$$

\end{thm}

L'ingr\'edient essentiel de la preuve est la formule de Grothendieck-Ogg-Shafarevich, elle-m\^eme d\'emontr\'ee par des arguments algébriques (voir \cite{R}), et dont on n'utilise ici que la forme modérée\footnote{Dans la direction oppos\'ee, il est int\'eressant de noter que la preuve initiale de la version mod\'er\'ee de cette formule reposait sur le th\'eor\`eme de structure du groupe fondamental, obtenu de mani\`ere transcendante (voir \cite{Milne} V Remark 2.19).}. Mentionnons enfin qu'une version nilpotente du th\'eor\`eme \ref{structure} a \'et\'e r\'ecemment prouv\'ee par Lieblich et Olsson (voir \cite{LO}).

\section{Preuves}
\subsection{D\'evissages}

On se ram\`ene au cas o\`u $A$ est un groupe ab\'elien
$l$-\'el\'ementaire $A$
(pour un nombre premier $l$ distinct de $p$), irr\'eductible pour l'action de $H$, gr\^ace aux deux lemmes suivants :
\begin{lem}
\label{devissageab}
Si la proposition \ref{state} est vraie pour tout groupe ab\'elien
$l$-\'el\'ementaire $A$, irr\'eductible pour l'action de $H$, elle est vraie
pour tout groupe ab\'elien $A$.
\end{lem}

\begin{lem}
\label{devissageres}
Si la proposition \ref{state} est vraie pour tout groupe ab\'elien
$A$, elle est vraie pour tout groupe r\'esoluble $A$.
\end{lem}

Introduisons quelques notations. Pour une suite exacte de groupes finis
$$(S) \quad 1\to A \to G \to H \to 1$$
notons $(\star ) _S$ la propri\'et\'e ${\cal P } _H \Rightarrow {\cal P}_G$.

La preuve des lemmes repose sur la remarque imm\'ediate suivante dont on omet la preuve :

\begin{lem} \label{outil}
On se donne la suite exacte 
$$(S) \quad 1\to A \to G \to H \to 1$$
et un sous-groupe $A_1$ de $A$ distingu\'e dans $G$ et on note $A_2 = A/A_1$. On dispose alors de deux nouvelles suites exactes 
$$(S_2) \quad 1 \to A_2 \to G/A_1 \to H \to 1$$
$$(S_1) \quad 1\to A_1 \to G \to G/A_1 \to 1$$
Alors si $(\star ) _{S_1}$ et $(\star ) _{S_2}$ sont vraies, il en est de m\^eme de $(\star )_S$.
\end{lem}

\subsubsection{Preuve du Lemme \ref{devissageab}}

\begin{proof}
Un groupe ab\'elien $A$ d'ordre premier \`a $p$ est le produit direct de ses $l$-sous-groupes de Sylow, qui sont caract\'eristiques, et donc distingu\'es dans $G$. On peut donc se ramener dans un premier temps, par un raisonnement par r\'ecurrence sur le nombre de Sylow de $A$, et  grâce au lemme \ref{outil} \`a un $l$-groupe ab\'elien $A$. Puis, en raisonnant par r\'ecurrence sur l'ordre de $A$, et en appliquant le lemme \ref{outil} \`a $A_1= A^l$, on est ramen\'e au cas de $l$-groupes ab\'eliens \'el\'ementaires. Enfin, une nouvelle r\'ecurrence sur le nombre de facteurs irr\'eductibles dans la d\'ecomposition de $A$ comme ${\mathbb F} _l [H]$-module, permet en utilisant encore le lemme \ref{outil} de se ramener au cas o\`u $A$ est irr\'eductible en tant que repr\'esentation de $H$.
\end{proof}

\subsubsection{Preuve du Lemme \ref{devissageres}}

\begin{proof}
La preuve est similaire \`a la pr\'ec\'edente. 
On raisonne par r\'ecurrence sur la longueur de la s\'erie d\'eriv\'ee de $A$, et en utilisant le lemme \ref{outil}.

\end{proof}

\subsection{Rel\`evement de revêtements d'une courbe affine}

On reprend les notations de l'introduction : on consid\`ere un corps alg\'ebriquement clos $k$ de caract\'eristique 
$p$ (\'eventuellement nulle), une courbe alg\'ebrique \emph{affine} lisse et connexe $X$ sur $k$, un point g\'eom\'etrique $\overline{x}$ .

On se donne un morphisme surjectif $\pi_1(X,\overline{x})\twoheadrightarrow H$ où $H$ est un groupe
fini et le revêtement galoisien connexe correspondant $\pi :Y \rightarrow X$, qu'on munit du point g\'eom\'etrique associ\'e $\overline{y}$ au dessus de $\overline{x}$. On fixe de plus un ${\mathbb F}_l[H]$-module $A$ de type fini irr\'eductible.

Comme $X$ est connexe, on a une \'equivalence de cat\'egorie classique entre syst\`emes locaux (pour la topologie \'etale) de ${\mathbb F}_l$-espaces vectoriels de dimension finie d'une part, et repr\'esentations de $\pi_1(X,\overline{x})$ \`a valeurs dans les ${\mathbb F}_l$-espaces vectoriels de dimension finie d'autre part. Cette \'equivalence envoie un faisceau localement constant $F$ sur sa fibre $F_{\overline{x}}$ en $\overline{x}$. De plus, $H^1(X, F)\simeq H^1(\pi_1(X,\overline{x}), F_{\overline{x}})$ (\cite{JPS}, preuve de la Proposition 1).

Ainsi au groupe $A$ vu comme repr\'esentation sur ${\mathbb F}_l$ de $\pi_1(X,\overline{x})$ \`a travers le morphisme $\pi_1(X,\overline{x})\twoheadrightarrow H$, est naturellement associ\'e un faisceau \'etale localement constant $\underline{A}$ trivialis\'e par $\pi :Y \rightarrow X$. Il s'agit simplement du faisceau étale $\pi_*^H(A_Y)$
donné explicitement par $(U\rightarrow X)\rightarrow(A_Y(\pi^{-1}U))^H$, où $A_Y$ est le faisceau constant sur $Y$ associ\'e au ${\mathbb F}_l$-espace vectoriel sous-jacent \`a $A$.

\begin{lem}
  \label{lem1}
  On a une suite exacte :
$$0 \rightarrow H^1(H,A) \rightarrow H^1(X,\underline{A}) \rightarrow
  \Hom_H(\pi_1^{ab}(Y), A) \rightarrow  H^2(H,A) \rightarrow
  0$$
 \end{lem}

 \begin{proof}
	 Il s'agit de la suite exacte courte de bas degr\'e associ\'ee \`a la suite spectrale de Hochschild-Serre $H^p(H,H^q(Y,A_Y))\Longrightarrow H^{p+q}(X,\underline{A})$ (voir appendice \ref{HS}). Pour identifier le troisi\`eme terme $H^1(Y,A_Y)^H$ de la suite on utilise l'isomorphisme $H^1(Y,A_Y)\simeq H^1(\pi_1(Y,\overline{y}),A)=\Hom(\pi_1(Y,\overline{y})^{ab},A)$, $A$ \'etant trivial comme $\pi_1(Y,\overline{y})$-module. L'annulation du cinqui\`eme terme, \`a savoir $H^2(X,\underline{A})$, r\'esulte de l'hypoth\`ese que $X$ est affine et donc de dimension cohomologique $1$ (\cite{SGA4} IX 5.7, X 5.2).
\end{proof}

\begin{lem}\label{lem2}
Le morphisme de transgression ${\rm trans}: {\rm
 Hom}_H(\pi_1^{ab}(Y),A) \rightarrow 
 H^2(H,A)$ dans la suite ci-dessus peut-être d\'ecrit comme suit : soit $u
 \in  {\rm Hom}_H(\pi_1^{ab}(Y),A)$, $u\neq 0$. Alors
 $u$ correspond canoniquement \`a un revêtement galoisien $Z\twoheadrightarrow X$
 dont le groupe de Galois est une extension de $H$ par $A$,
 et ${\rm trans}(u)$ est l'oppos\'e de la classe de cette extension dans
 $H^2(H,A)$.
 \end{lem}

 \begin{proof}
Comme la dimension cohomologique de $\pi_1(X,\overline{x})$ est $1$ (\cite{JPS}, Proposition 1), on peut, alternativement, consid\'erer la suite exacte du lemme \ref{lem1} comme la suite exacte d'inflation-restriction relative au module $A$ et \`a la suite exacte $1\rightarrow  \pi_1(Y,\overline{y})\rightarrow  \pi_1(X,\overline{x})\rightarrow H\rightarrow 1$. On conclut grâce au fait suivant (pour lequel on renvoie \`a \cite{HS}, qui traite le cas des groupes abstraits) : soit $1 \rightarrow {\bf F} \rightarrow {\bf G} \rightarrow {\bf H} \rightarrow 1$ une suite exacte de groupes profinis, $A$ un groupe ab\'elien discret \'equip\'e d'une action continue de ${\bf H}$.  Soit de plus ${\bf F'}$ le sous-groupe d\'eriv\'e de ${\bf F}$, et $\gamma \in H^2({\bf H},
	 {\bf F}/{\bf F'})$ la classe de l'extension :
$ 1 \rightarrow {\bf F}/{\bf F'} \rightarrow  {\bf G}/{\bf F'} \rightarrow {\bf H} \rightarrow 1$. Alors le morphisme de transgression  ${\rm trans} : H^1({\bf F}, A)^{\bf H} \rightarrow H^2({\bf H}, A)$ 
(c'est-\`a-dire le quatri\`eme morphisme dans la suite exacte d'inflation-restriction) est donn\'e explicitement par :
$$ \forall u \in  H^1({\bf F}, A)^{\bf H}=H^0({\bf H}, {\rm Hom}({\bf
  F}/{\bf F'}, A))\;\;\;\; {\rm trans}(u) = -\gamma \cup u $$

 \end{proof}

\subsection{Preuve de la Proposition \ref{state}}

Supposons ${\cal P}_H$
vraie. Le cas où les deux assertions de ${\cal P}_H$ sont fausses est \'evident.
On supposera donc, dans ce paragraphe, que ces deux affirmations sont vraies \`a savoir $n_H\leq 2g+r-1 $ et $ \exists \pi_1(X,\underline{x})\twoheadrightarrow H$. Il faut v\'erifier ${\cal P}_G$ :
$$n_G\leq 2g+r-1 \iff \exists \pi_1(X,\underline{x})\twoheadrightarrow G$$

\subsubsection{Preuve de la Proposition \ref{state} lorsque $G$ n'est pas produit semi-direct}

La classe de la suite exacte $1\to A \to G \to H \to 1$ dans $H^2 (H,A)$ n'est pas triviale, et son oppos\'ee se rel\`eve donc en un \'el\'ement non trivial de $Hom _H ( \pi _1^{ab} (Y), A)$
, qui est surjectif car $A$ est irr\'eductible
; ce dernier correspond \`a un rev\^etement galoisien 
connexe
$Z\to Y$ de groupe $A$, tel que le compos\'e $Z\to Y \to X$ soit galoisien de groupe $G$ et tel que la suite exacte correspondante de groupes de Galois soit la suite exacte de d\'epart d'apr\`es le lemme \ref{lem2}. Cela montre en particulier que la propri\'et\'e $\exists \pi_1(X,\underline{x})\twoheadrightarrow G$ est toujours vraie.

Il s'agit donc de v\'erifier que 
$n_G\leq 2g+r-1$ 
est vrai. Cela r\'esulte du fait que 
$n_G= n_H\leq 2g+r-1$ car $G$ n'est pas produit semi-direct. 
En effet, soient $\tilde{x_1},\cdots,\tilde{x_{n_H}}$ des relev\'es dans $G$ d'un syst\`eme minimal de g\'en\'erateurs $x_1,\cdots,x_{n_H}$ de $H$, et $\tilde{G}<G$ le sous-groupe qu'ils engendrent. Le sous-groupe $A\cap \tilde{G}$ de $A$ est stable pour l'action de $H$, et du fait de l'irr\'eductibilit\'e de $A$, $A\cap \tilde{G} = A$ ou $A\cap \tilde{G} = 1$. Cette derni\`ere \'egalit\'e est impossible du fait que $G$ n'est pas produit semi-direct. Donc $A<\tilde{G}$ et $\tilde{G}=G$.

\subsubsection{Preuve de la Proposition \ref{state} lorsque $G$ est produit semi-direct : pr\'eliminaire}

Si $G$ est produit semi-direct de $A$ par $H$, on d\'eduit du lemme \ref{lem1} le lemme suivant :

\begin{lem}
  \label{lem3}
  Le probl\`eme de plongement
$$\xymatrix{ &&& \pi_1(X,\underline{x}) \ar@{->>}[d]\ar@{-->}[dl] \\1 \ar[r] & A \ar[r] & G \ar[r]^{q} &  H \ar[r]  & 1  }  $$     
a une solution forte \footnote{On rappelle que cela signifie qu'il existe un \'epimorphisme $\pi_1(X,\underline{x})\rightarrow G$ faisant commuter le diagramme donn\'e.} si et seulement si
$\dim_{{\mathbb F}_l}H^1(X,\underline{A})>\dim_{{\mathbb F}_l}H^1(H,A)$.
\end{lem}

\begin{proof} Une solution forte du probl\`eme de plongement correspond en effet \`a un \'el\'ement non trivial de $Hom_H( \pi _1 ^{ab} (Y), A)$ s'envoyant sur la classe triviale de $H^2 (H,A)$. 
D'apr\`es le lemme \ref{lem1}, un tel \'el\'ement existe si et seulement si $\dim_{{\mathbb F}_l}H^1(X,\underline{A})>\dim_{{\mathbb F}_l}H^1(H,A)$.
\end{proof}
\subsubsection{Calcul de $\dim_{{\mathbb F}_l}H^1(X,\underline{A})$}

\begin{lem}
  \label{lem4}
Si 
$H$ est un $p'$-groupe alors :
$$\dim_{{\mathbb F}_l}H^1(X,\underline{A})=(2g+r-2)\dim_{{\mathbb
    F}_l}A+\dim_{{\mathbb F}_l}A^H\; .$$
\end{lem}

\begin{proof}
	C'est une cons\'equence imm\'ediate du th\'eor\`eme de Grothendieck-Ogg-Shafarevich : voir le corollaire \ref{goscor} de l'appendice \ref{gosapp}. En effet en notant $\overline{\pi} : \overline{Y}\rightarrow \overline{X}$ l'unique revêtement de courbes projectives et lisses prolongeant $\pi$, qui est encore galoisien de groupe $H$, on peut appliquer la formule au faisceau $F=\overline{\pi}_*^H{A_{\overline{Y}}}$, qui est constructible car $\overline{\pi}$ est propre, et dont la restriction \`a $X$ est $F_{|X}=\pi_*^H A_{Y}=\underline{A}$. L'hypoth\`ese sur $H$ fait que $F$ est mod\'er\'ement ramifi\'e, en d'autres termes pour tout $x\in \overline{X}(k)$ on a $\alpha_{x}(F)=0$, d'où la formule donn\'ee.

\end{proof}

\subsubsection{Lemme alg\'ebrique}
On fixe un entier naturel $N$.

\begin{lem}
  \label{lem5} On suppose que $A$ est un groupe ab\'elien
$l$-\'el\'ementaire irr\'eductible
\footnote{En particulier $A$ est suppos\'e non trivial.} pour l'action de $H$ et que $n_H\leq N$.  On note $G$ le produit semi-direct 
$G=A\rtimes H$.

Alors :
$$\dim_{{\mathbb F}_l}H^1(H,A)<(N-1)\dim_{{\mathbb F}_l}A+\dim_{{\mathbb F}_l}A^H \iff n_G\leq N\; . $$

\end{lem}

\begin{proof}
On note ${\cal S}=\{s:H\rightarrow G\}$ l'ensemble des sections (qui sont
  des morphismes de \emph{groupes}) de la projection canonique
  $q:G\twoheadrightarrow H$. Le groupe $A$ agit sur $\cal S$ par conjugaison, et
  il est bien connu que l'ensemble quotient est en bijection avec $H^1(H,A)$.
Comme le stabilisateur de \emph{toute} section est $A^H$, on en d\'eduit
  que

  $$|A^H| \cdot|{\cal S}| =|H^1(H,A)|\cdot|A|$$

  Il s'agit donc de montrer :

  $$n_G\leq N \iff  |{\cal S}|<|A|^N\; .$$

\noindent On suppose d'abord 
$n_G\leq N$, disons $<u_1,\cdots,u_N>=G$. Alors 
$<q(u_1),\cdots,q(u_N)>=H$
, et donc toute section est d\'etermin\'ee uniquement par
ses valeurs sur 
$q(u_1),\cdots,q(u_N)$, qui sont du type $a_1u_1,\cdots,a_Nu_N$, avec $(a_1,\cdots, a_N)\in A^N$.  Il y a donc au plus 
$|A|^N$ 
sections, mais le choix
$(1,\cdots,1)$
ne convenant pas, on a donc  
$|{\cal S}|<|A|^N$.

R\'eciproquement on suppose 
$n_G>N$. Soit $\{x_1, \cdots , x_N\} $ un 
syst\`eme g\'en\'erateur de $H$ et $\tilde{x_1},\cdots,\tilde{x_N}$ des rel\`evements arbitraires dans $G$. Alors pour tout $(a_1,\cdots, a_N)\in A^N$
on va montrer qu'il existe une (unique) section $s$ telle que $s(x_1)=a_1 \tilde{x_1}, \cdots, s(x_N)=a_N \tilde{x_N}$. Comme deux telles sections sont distinctes, ceci impliquera $|{\cal S}|\geq|A|^N$.

Pour prouver l'existence de la section, on pose $\tilde{G}= <a_1 \tilde{x_1},\cdots,  a_N \tilde{x_N}>\subset G$. 
Alors $\tilde{G}$ se surjecte sur $H$ et
$B=A\cap \tilde{G}$ est un sous $H$-module de $A$. Comme $A$ est irr\'eductible
et 
$n_G>N$
on doit avoir $B=1$, ce qui montre que 
$q$ induit un
isomorphisme de $\tilde{G}$ sur $H$, et l'isomorphisme r\'eciproque fournit
la section souhait\'ee\footnote{Remarquons qu'il ressort de la preuve que sous l'hypoth\`ese $n_G>N\geq n_H$, on a $n_G=n_H+1$ et $N=n_H$. Cela peut \^etre v\'erifi\'e directement d'ailleurs par le m\^eme genre d'arguments : il suffit en effet de voir que $n_G\leq n_H+1$. Soit $x_1,\cdots,x_{n_H}$ un syst\`eme minimal de g\'en\'erateurs de $H$, qu'on rel\`eve arbitrairement dans $G$ en $\tilde{x_1},\cdots,\tilde{x_{n_H}}$. Soit de plus $a\neq 0$ dans $A$. Alors $\tilde{x_1},\cdots,\tilde{x_{n_H}},a$ engendrent $G$~: en effet si $\tilde{G}$ est le sous-groupe de $G$ engendr\'e par ces \'el\'ements, alors  $\tilde{G}$ s'envoie sur $H$ de mani\`ere surjective, le noyau $A\cap \tilde G$ est muni d'une action de $H$ compatible avec son inclusion dans $A$, et contient $a$, donc par irr\'eductibilit\'e doit \^etre \'egal \`a $A$, ce qui impose $\tilde{G}=G$.}.
\end{proof}

\subsubsection{Preuve de la proposition \ref{state} lorsque $G$ est produit semi-direct : conclusion}
On utilise les lemmes  \ref{lem3}, \ref{lem4} et \ref{lem5} (le dernier avec $N=2g+r-1$). Plus pr\'ecis\'ement, pour montrer l'implication $n_G\leq 2g+r-1 \implies \exists \pi_1(X,\underline{x})\twoheadrightarrow G$, on part d'un \'epimorphisme $\pi_1(X,\underline{x})\twoheadrightarrow H$, dont on a suppos\'e l'existence, et on le rel\`eve \`a $G$ gr\^ace à l'hypoth\`ese $n_G\leq 2g+r-1$ et aux lemmes. Pour montrer l'implication réciproque, on part alors d'un \'epimorphisme $\pi_1(X,\underline{x})\twoheadrightarrow G$, puis on applique le lemme \ref{lem3} à l'\'epimorphisme compos\'e $\pi_1(X,\underline{x})\twoheadrightarrow G\twoheadrightarrow H$, et les lemmes \ref{lem4} et \ref{lem5} montrent qu'alors $n_G\leq 2g+r-1$. 

\subsection{Preuve du th\'eor\`eme \ref{structure}}

	La proposition \ref{state} montre que les deux groupes profinis $\pi_1^{res}(X,\overline{x})$ et  $F_{2g+r-1} ^{res} $ ont mêmes quotients finis. Comme le second est topologiquement de type fini, cela suffit d'apr\`es \cite{FJ}, Proposition 15.4 ou  \cite{FJ2}, Proposition 16.10.7. .

\section{Remarque sur le cas des groupes d'ordre divisible par $p$}

Il semble naturel de demander ce qu'il subsiste de la m\'ethode pr\'ec\'edente pour des groupes d'ordre divisible par la caract\'eristique $p$ du corps $k$. En particulier, on aimerait avoir une preuve alg\'ebrique de la conjecture d'Abhyankar\footnote{On rappelle que celle-ci affirme qu'un groupe fini est groupe de Galois d'une extension \'etale d'une courbe affine $X$ si et seulement si son quotient par le sous-groupe engendr\'e par ses $p$-sous-groupes de Sylow admet au plus $2g+r-1$ g\'en\'erateurs. Cette conjecture a \'et\'e d\'emontr\'ee par Raynaud pour la droite affine \cite{R2} et Harbater dans le cas g\'en\'eral \cite{Har}.}, pour les groupes r\'esolubles pour toute courbe affine, g\'en\'eralisant d'une part celle donn\'ee dans \cite{JPS} pour la droite affine, et le th\'eor\`eme \ref{structure} d'autre part.

Dans cette direction, on remarque que l'hypoth\`ese que $H$ est un $p'$-groupe n'intervient que dans le lemme \ref{lem4}. Si on la remplace par l'hypoth\`ese plus faible que l'action de $H$ sur le faisceau $F$ introduit dans la preuve du lemme \ref{lem4} est mod\'er\'ee, on a obtenu au passage le r\'esultat suivant.

\begin{prop}

Soit $\pi_1(X,\overline{x})\twoheadrightarrow H$ où $H$ est un groupe
fini d'ordre quelconque, et $\pi :Y \rightarrow X$ le revêtement galoisien connexe correspondant.

On consid\`ere le probl\`eme de plongement suivant :

$$\xymatrix{ &&& \pi_1(X,\underline{x}) \ar@{->>}[d]\ar@{-->}[dl] \\1 \ar[r] & A \ar[r] & G \ar[r]^{q} &  H \ar[r]  & 1  }  $$  

Si $A$ est ab\'elien $l$-\'el\'ementaire avec $l\neq p$, et irr\'eductible pour l'action de $H$, si de plus le faisceau \'etale $\underline{A}$ associ\'e sur $X$ est mod\'er\'e, si enfin $n_G\leq 2g+r-1$, alors le probl\`eme consid\'er\'e a une solution forte.  

\end{prop}

\begin{rem}
	L'hypoth\`ese que $\underline{A}$ est mod\'er\'e est bien sûr v\'erifi\'ee si le prolongement canonique de $\pi$ en un revêtement de courbes projectives lisses est lui-m\^eme mod\'er\'e. Dans cette situation, les m\'ethodes transcendantes et alg\'ebriques ne sont pas vraiment concurrentes, mais compl\'ementaires :  le th\'eor\`eme d'existence de Riemann, alli\'e \`a la th\'eorie de la sp\'ecialisation du groupe fondamental,  affirme (\cite{SGA1} XIII Corollaire 2.12) que la condition $n_G\leq 2g+r-1$ est n\'ecessaire pour que $G$ soit un quotient du groupe fondamental mod\'er\'e, et la technique de rel\`evement adopt\'ee ici montre, pour ce type tr\`es particulier de groupe, qu'elle est aussi suffisante.  
\end{rem}

\appendix

\section{Appendices}

\subsection{Hochschild-Serre}
\label{HS}
\begin{thm}[Suite spectrale de Hochschild-Serre]
  Soit $\pi : Y \rightarrow X$ un revêtement galoisien de groupe $H$,
  et $F$ un faisceau (ab\'elien) \'etale sur $X$. Alors on a une suite
  spectrale :
$$H^p(H,H^q(Y,\pi^* F))\Longrightarrow H^{p+q}(X,F)$$
\end{thm}

\begin{proof}
	Voir par exemple \cite{Milne} III Theorem 2.20.
	
\end{proof}

Cette suite spectrale \'etant cohomologique, on dispose en particulier de la suite exacte courte des termes de bas degr\'e, qui s'\'ecrit ici :

\begin{displaymath}
	0\rightarrow H^1(H,H^0(Y,\pi^* F))\rightarrow H^{1}(X,F)\rightarrow H^1(Y,\pi^* F)^H\rightarrow  H^2(H,H^0(Y,\pi^* F))\rightarrow  H^{2}(X,F)
\end{displaymath}

\subsection{Grothendieck-Ogg-Shafarevich}
\label{gosapp}
On suit \cite{R}, \cite{Milne} : $k$ est un
corps alg\'ebriquement clos de caract\'eristique
$p$ (\'eventuellement nulle), $\overline{X}$ une courbe alg\'ebrique irr\'eductible, projective et lisse sur $k$, de point g\'en\'erique $\overline{\nu}$.

\begin{defi}
Soit $l$ un nombre premier distinct de $p$, et $F$ un faisceau
constructible de ${\mathbb F_l}$-modules sur $\overline{X}$. \emph{L'invariant de Swan de $F$ en $x\in \overline{X}(k)$} est l'entier :

$$\alpha_x(F)=\sum_{i=1}^{\infty} \frac{g_i}{g_0}
\dim_{{\mathbb F}_l} F_{\overline{\nu}}/F_{\overline{\nu}}^{G_i}$$

où $g_i$ est l'ordre de $G_i$, $i$-\`eme groupe de ramification en $x$
dans un revêtement galoisien $\overline{Y}/\overline{X}$ trivialisant $F$ au point g\'en\'erique.
\end{defi}

\begin{defi}\emph{L'exposant du conducteur de $F$ en $x$} est :
$$\epsilon_x(F)=\alpha_x(F)+\dim_{{\mathbb
      F}_l}F_{\overline{\nu}}-\dim_{{\mathbb F}_l}F_x$$ 
\end{defi}

On a alors :

\begin{thm}[Grothendieck-Ogg-Shafarevich]
$$\chi(\overline{X},F)=\chi(\overline{X})\dim_{{\mathbb F}_l}F_{\overline{\nu}}-\sum_{x\in \overline{X}(k)}
\epsilon_x(F)\; .$$
\end{thm}

\begin{cor}
	\label{goscor}
  Soit $X$ un ouvert affine de $\overline{X}$, ne contenant aucun des points où $F$ est ramifi\'e. Alors :

  $$\chi(X,F_{|X})=\chi(X)\dim_{{\mathbb
  F}_l}F_{\overline{\nu}}-\sum_{x\in \overline{X}(k)} \alpha_x(F)\; .$$

  \end{cor}

Le corollaire est imm\'ediat \`a partir de la suite de cohomologie
relative de la paire $(\overline{X},X)$ et du fait que $\dim_{{\mathbb
    F}_l}F_x=\sum_i(-1)^i \dim_{{\mathbb
    F}_l} H^i_x(\overline{X},F)$ (\cite{Milne} V Lemma 2.10 (a)).

\begin{rem}
Notons que si 
  $r$ est le nombre de points de $\overline{X}-X$, et $g$ est le genre de $\overline{X}$, alors 
  $\chi(X)=2-2g-r$.
\end{rem}

\end{document}